\documentclass{lms}
\usepackage{amsfonts,amssymb,amsmath,amsopn,amstext,amsbsy}
\usepackage[cp1251]{inputenc} 
\allowdisplaybreaks[4]
\usepackage[active]{srcltx} 

\newtheorem{thm}{{{Theorem}}}[section] 
\newtheorem{lemma}{Lemma}[section]     

\newnumbered{definition}{Definition}
\newnumbered{assertion}{Assertion}

\newnumbered{remark}{Remark}
\newtheorem{thmx}{{{Theorem}}}

\newtheorem{lemx}{Lemma}

\newtheorem{conx}{Conjecture}

\DeclareSymbolFont{AMSb}{U}{msa}{m}{n} \DeclareMathAlphabet{\msa}{U}{msa}{m}{n}
\DeclareSymbolFont{AMSb}{U}{msb}{m}{n} \DeclareMathAlphabet{\Bb}{U}{msb}{m}{n}
\DeclareSymbolFont{AMSb}{U}{eus}{m}{n} \DeclareMathAlphabet{\eusm}{U}{eus}{m}{n}
\DeclareSymbolFont{AMSb}{U}{euf}{m}{n} \DeclareMathAlphabet{\eufm}{U}{euf}{m}{n}
\DeclareSymbolFont{AMSb}{U}{eur}{m}{n} \DeclareMathAlphabet{\eurm}{U}{eur}{m}{n}

 \newcommand{\C}{\Bb{C}}
 
\DeclareMathOperator{\hol}{\rm{Hol}}

\usepackage{hyperref}

\title[More properties of the Ramanujan sequence]{{\bf{More
  Properties of the Ramanujan Sequence}}}

\author{Andrew  Bakan, Stephan Ruscheweyh and  Luis Salinas}%
 \classno{33B10, 30C45  (primary),  26D07 (secondary)}

\extraline{The first author was supported by the German Academic
  Exchange Service (DAAD, Grant 57210233). The second and third
  authors acknowledge partial support from FONDECYT, Grant 1150810, and the Centro
  Cient\'\i fico Tecnol\'ogico de Valpara\'\i{}so -- CCTVal. }

\begin{document}
\maketitle

\begin{abstract}
The  Ramanujan sequence
$\{\theta_{n}\}_{n \geq 0}$, defined as
\begin{gather*}\theta_{0}= \frac{1}{2} \ , \ \ \
    \theta_{n} = \Big(\ \ \frac{e^{n}}{2} - \sum_{k=0}^{n-1} \frac{n^{k}}{k !} \ \
    \Big) \cdot   \frac{n !}{n^{n}}  \ , \ \ n \geq 1 \ ,
\end{gather*} has been studied on many occasions and in many different
    contexts.  J. Adell and P. Jodra \cite{Ad}(2008) and S. Koumandos \cite{K}(2013) showed, respectively, that the
    sequences $\{\theta_{n}\}_{n \geq 0}$  and  $\{4/135 - n \cdot (\theta_{n}- 1/3 )\}_{n \geq 0}$ are completely
    monotone. In the present paper we establish  that the sequence $\{(n+1) \cdot
(\theta_{n}- 1/3 )\}_{n \geq 0}$  is also completely monotone. Furthermore, we prove that the analytic
    function
$(\theta_{1}- 1/3 )^{-1}\sum_{n=1}^{\infty} (\theta_{n}- 1/3 ) \cdot z^{n} / n^{\alpha}$
is universally starlike for every $\alpha \geq 1$ in the slit domain
    $\C\setminus[1,\infty)$.  This seems to be the first result putting the
    Ramanujan sequence into the context of analytic univalent
    functions and is a step towards a previous stronger
conjecture, proposed by S. Ruscheweyh, L. Salinas and T. Sugawa in
    \cite{RSS}(2009),  namely that
 the function $(\theta_{1}- 1/3 )^{-1}\sum_{n=1}^{\infty} (\theta_{n}- 1/3 ) \cdot z^{n} $ is universally convex.

\end{abstract}
\maketitle
\section{Introduction}

\vspace{0.25cm}
A famous problem raised by Ramanujan in \cite{R2}(1911) states that
 the so-called Ramanujan numbers
$\theta_{n}$, $n \geq 0$, defined as
\begin{gather}\label{1}\theta_{0}= \frac12 \ , \ \ \
    \theta_{n} = \Big(\ \ \frac{e^{n}}{2} - \sum_{k=0}^{n-1} \frac{n^{k}}{k !} \ \
     \Big) \cdot   \frac{n !}{n^{n}}  \ , \ \ n \geq 1 \ ,
\end{gather}
satisfy
\begin{gather}\label{2}
    \theta_{n}  \in  \left[\  \frac{1}{3} \  , \ \frac{1}{2} \ \right] \ .
\end{gather}

\noindent In his first letter to Hardy \cite{R2}(1913)
Ramanujan refined his conjecture \eqref{2} as follows
\begin{gather}\label{3}
   \theta_{n} =  \frac{1}{3} + \frac{4}{135} \cdot \frac{1}{n + k_n} , \ \
   k_n \in  \left[\ \frac{2}{21} \ , \  \frac{8}{45} \ \right] \  , \ \ n \geq 0 \  .
\end{gather}

\noindent
The first proofs of \eqref{2} were published by G.Szeg{\"o} \cite{Sz} (1928) and
 G.N.Watson \cite{W}(1929). A proof of \eqref{3} was obtained in 1995
 by Flajolet et al. \cite{Fl}.

In 2003 S.E.Alm \cite{Alm} showed that the sequence $\{k_{n}\}_{n \geq 0}$ appearing in \eqref{3} is decreasing.
In 2008 J. Adell and P. Jodra \cite{Ad} proved that  there is a probability
distribution function $F$ on $[0,1]$ such that
\begin{gather}\label{4}
   \theta_n - 1/3 = \frac{1}{6} \int_{0}^{1} x^{n} d \, F (x)  \ , \
   n \geq 0 \ ,
\end{gather}
which implies that the sequence $\{\theta_{n}\}_{n \geq 0}$  is completely monotone, i.e.
\begin{gather}\label{5}
\sum_{m=0}^{n}\binom{n}{m} (-1)^{m} \theta_{k+m} \geq 0 \ , \ \ k \geq 0 \ , \ n \geq 0 \ .
\end{gather}

In 2013 S. Koumandos \cite{K}
proved the existence of a strictly positive  function $k$ on $[0, +\infty)$  such that
\begin{gather}\label{5a}
\frac{4}{135} -  n \cdot \left( \theta_{n} - \frac{1}{3}\right)  =
\frac{1}{2} \int_{0}^{\infty} e^{- n x} k (x)  d x \ ,  \ n \geq 0 \ ,
\end{gather}

\noindent
and noted that the complete monotony of the  sequence
$\{4/135 - n \cdot (\theta_{n}- 1/3 )\}_{n \geq 0}$ follows from \eqref{5a}.

%

\smallskip
 We refer the reader to Alzer \cite{Al}(2004), J. Adell and P. Jodra \cite{Ad}(2008)  and S. Koumandos \cite{K}(2013)
  for their surveys  of other previous results on the Ramanujan sequence.

\section{The results}

\subsection{More monotonicity properties}

In this paper we refine the property \eqref{4} as follows.
\begin{thm}\label{th1}
   There is a probability distribution function $G$ on $[0,1]$ such that
   \begin{gather}\label{6}
 \left(n + 1 \right)  \left(\theta_n - 1/3\right)   =     \frac{4}{135} +
   \frac{37}{270} \int_{0}^{1} x^{n} d \, G (x)  \ , \
   n \geq 0 \ .
\end{gather}
As a consequence, the sequence
\begin{gather*}
\Big\{ \ \left(n + 1 \right)  \left(\theta_n - 1/3\right) \ \Big\}_{n \geq 0}
\end{gather*}
 is completely monotone.
\end{thm}

We also show in Section~\ref{profth1} that Theorem \ref{thm4} below yields easily \eqref{4} and the validity of  \eqref{5a}  written in the following equivalent form
 \begin{gather}\label{7}\frac{4}{135} -  n \cdot \left( \theta_{n} - \frac{1}{3}\right)
  = \frac{4}{135} \int_{0}^{1} t^{ n }d D (t) \ , \ n = 0, 1, 2, ... \  ,
\end{gather}

\noindent
where $D$ is a continuous probability distribution function  on $[0,1]$.

\subsection{The Ramanujan sequence and univalent functions}

Let $\Lambda$ denote the slit domain $\C\setminus[1,\infty)$ and
$\hol(\Lambda)$ the set of analytic functions in $\Lambda$. We write
$f\in\hol_1(\Lambda)$ if $f\in\hol(\Lambda)$ satisfies $f(0)=f'(0)-1=0$.
The following definition has been introduced in \cite[Def.1.3, 1.4,
pp. 290--291]{RSS}.

\begin{definition}
  A function $f\in\hol_{1}(\Lambda)$ is called
  universally starlike if it maps every circular domain
  $\Omega\subset\Lambda$ with $0\in\Omega$ conformally onto a domain
  starlike with respect to the origin. It is called universally convex
  if it maps every circular domain $\Omega\subset\Lambda$ conformally
  onto a convex domain.
\end{definition}

Note that in this definition circular domains are meant to be open
 disks or open half-planes in $\C$. It is an immediate consequence of
 the definition that
 \begin{equation}
   \label{eq:x1}
   f \mbox{ universally convex} \Rightarrow f  \mbox{ universally starlike} \ .
 \end{equation}

In \cite[p.294]{RSS} the following conjecture has been proposed.
\begin{conx}[(S. Ruscheweyh,  L. Salinas,  and T. Sugawa, 2009 {\cite{RSS}})]
The function
\begin{gather}\label{10}
  \sigma (z) := \frac{1}{\theta_{1}- 1/3} \cdot \sum_{n=1}^{\infty} (\theta_{n}- 1/3 ) \cdot z^{n}
\end{gather}
is universally convex.
\end{conx}

In a recent paper \cite{BRS} the present authors established the
  following general result.

\begin{thmx}\label{poly}
 For $f(z)=\sum_{n=1}^{\infty}a_n z^n\in\hol(\Lambda)$
  let $f_{\alpha}(z):=\sum_{n=1}^{\infty}n^{-\alpha} a_n z^n$.
  Then we  have:

 \noindent
 1. If
  $f$ is universally convex then the functions $f_{\alpha},\ \alpha\geq1$, are also
  universally convex.

  \noindent
2.  If
  $f$ is universally starlike then the functions $f_{\alpha},\ \alpha\geq0$, are also
  universally starlike.
   \end{thmx}

We shall prove

\begin{thm}\label{th3}
   The functions
    \begin{gather*}
  \sigma_\alpha(z):= \frac{1}{\theta_{1}- 1/3} \sum_{n=1}^{\infty}
(\theta_n - 1/3) \frac{z^{n}}{n^{\alpha}} \ ,
    \end{gather*}
    are universally starlike for every $\alpha \geq 1$.
\end{thm}

In view of Theorem \ref{poly} and \eqref{eq:x1}  it is clear that Theorem \ref{th3} represents a necessary
condition for Conjecture A to be valid, and therefore is a first step
towards the still open decision concerning Conjecture A.

\section{Watson’s approach}\label{s2}

\vspace{0.25cm}
In this section we follow Watson's reasonings from \cite{W}.
On the positive half-line there exist two functions $u$ and $U $ satisfying the following relations
\begin{gather}\label{2.1}
    u (x) \cdot e^{{1 - u (x) }} = e^{-x} \ , \ \ \
    U (x) \cdot e^{{1 - U (x) }} = e^{-x} \ ,  \ \  0 \leq u (x) \leq 1\leq U (x) \  , \ \
    x \geq 0 \ .
\end{gather}
The function  $ U $ is strictly increasing on $[0, +\infty)$ with $ U (0)=1$ and $U(x)\to +\infty$ as $x \to +\infty$, whereas $ u $ is strictly decreasing on $[0, +\infty)$ with $ u(0)=1$ and $\lim_{x \to +\infty} u(x)=0$. Furthermore, $ u $ and $ U $ satisfy  the differential equations (see \cite[p.295]{W})
\begin{gather}\label{2.2}
   U^{\, \prime } (x)  = \frac{U  (x)}{U (x) - 1} \ , \ \
     u^{\, \prime } (x)  = - \frac{u (x)}{1 - u (x) } \ , \ x > 0 \   ,
\end{gather}
which  imply that
 \begin{gather}\label{2.3}
\begin{array}{ll} U^{\, \prime \prime  } (x)  = - \dfrac{U(x)}{\left(U(x) - 1\right)^{3}}\ , &
u^{\, \prime \prime  } (x)  = \dfrac{u(x)}{\left(1 - u(x)\right)^{3}} \ ,\\[0.5cm]
U^{\, \prime \prime \prime   } (x)  = \dfrac{U (x) (2 U(x) + 1)}{\left(U(x) - 1\right)^{5}}\ , &
u^{\, \prime \prime \prime   } (x)  = - \dfrac{u (x) (2 u(x) + 1)}{\left(1 - u (x)\right)^{5}} \ .
\end{array}
\end{gather}
\noindent
If we put
\begin{gather}\label{2.4}
    w (x) :=  u (\log x) \ ,  \ \  W (x) :=  U (\log x) \ ,  \ x \geq 1 \ ,
\end{gather}
\noindent
then
\begin{gather}\label{2.5} \frac{e^{{ w (x) }}}{ e \cdot w (x)}
    = x \ , \ \ \frac{e^{{ W (x) }}}{ e \cdot W (x)}
    = x \ ,
     \ \  0 \leq w (x) \leq 1\leq W (x) \  , \ \
    x \geq 1 \ ,
\end{gather}
and  therefore
\begin{gather}\label{2.6}
    w \left( \frac{e^{x}}{e \cdot x}\right) = x \ ,  \ \ x \in [0,1] \ ,
\ \
   W \left( \frac{e^{x}}{e \cdot x}\right) = x \ , \ \  x \geq 1 \ .
   \end{gather}
Since, on the positive semiaxis, the function $x / (e^{x} -1)$ decreases from $1$ to zero
while  $x / (1 - e^{- x})$ increases from $1$ to $+\infty $ the
equations  \eqref{2.6} can be written as
\begin{gather}\label{2.7}
     w \left( \dfrac{ \exp\left(\ \dfrac{x}{e^{x} -1 }\right)}{e \cdot \dfrac{x}{e^{x} -1 }}\right) = \frac{x}{e^{x} -1 } \ , \ \
W \left( \frac{\exp\left( \ \dfrac{x}{1 - e^{- x} }\right)}{e \cdot \dfrac{x}{1 - e^{- x} }}\right) = \frac{x}{1 - e^{- x} }\ , \ \ x \geq 0 \ .
\end{gather}

Elementary calculations show that the even function
\begin{equation}\label{2.8}
  \rho (x) :=  \frac{  \exp\left(\ \dfrac{x}{e^{x} -1 }\right)   }{e \cdot \dfrac{x}{e^{x} -1 }} =
\dfrac{ \exp\left( \ \dfrac{x}{1 - e^{- x} }\right) }{e \cdot \dfrac{x}{1 - e^{- x} }} \ , \ \ x \in \Bb{R} \ ,
\end{equation}
\noindent
satisfies $\rho(0)=1$ and
\begin{gather}
\label{2.9}
\frac{\rho^{\, \prime } (x)}{\rho (x)} =
\dfrac{e^{x} (e^{x}\!-\! 1\! -\! x) (e^{-x}\!-\! 1\! +\! x) }{x
  (e^{x}\! -\! 1)^{2}  }\ > 0,\quad x>0.
\end{gather}
\noindent Therefore \eqref{2.8}  and \eqref{2.4} imply
\begin{equation}\label{2.11}
U \left( \log  \rho (x)\right) = \frac{x}{1 - e^{- x} }=:H(x)
  \ , \ \  u \left( \log \rho (x)\right) = \frac{x}{e^{x} -1 }=:h(x)
\ , \ \ x \in \Bb{R} \ .
\end{equation}
and by virtue of \eqref{2.3} we obtain for arbitrary $x > 0$ the representations

\begin{align}
\label{2.12}
U^{\, \prime } \left( {{\log \rho (x)}}\right)
+ u^{\, \prime } \left( {{\log \rho (x)}}\right) & =
\frac{H(x)}{H(x)-1} +\frac{h(x)}{h(x)-1}  \ , \\[0.3cm]
\label{2.13}
\ U^{\,\prime \prime } \left( {{\log \rho (x)}}\right)
+ u^{\,\prime \prime } \left( {{\log \rho (x)}}\right)
& =\frac{H(x)}{(1-H(x))^3}+\frac{h(x)}{(1-h(x))^3}\\[0.3cm]
\label{2.14}
U^{\,\prime \prime \prime  } \left( {{\log \rho (x)}}\right)
+ u^{\,\prime \prime \prime  } \left( {{\log \rho (x)}}\right)
 & =  H(x)\frac{1+2H(x)}{(H(x)-1)^5} + h(x)\frac{1+2h(x)}{(h(x)-1)^5}
\  \ ,
\end{align}
\vspace{3mm}

The following key result concerning these quantities will be established in Section~\ref{sect5}.
\begin{thm}
  \label{thm4}
For $x>0$ we have
\begin{equation}
  \label{eq:x2}
  U^{\, \prime }(x)+u^{\, \prime }(x)>-(U^{\, \prime \prime }(x)+u^{\, \prime \prime }(x))>
  U^{\, \prime \prime \prime  }(x)+u^{\, \prime \prime \prime  }(x)>0.
\end{equation}
\end{thm}

The last inequality in \eqref{eq:x2} is known (see Koumandos \cite[Lemma 2, p.452]{K}). However, in order
to make the present paper more self-contained. a new proof, based on the new algorithm disclosed in
~\ref{sect51}, is  given in Subsection~\ref{thep}.


\section{\texorpdfstring{Proof of Theorem \ref{th1} }{Proof of Theorem 2.1}}\label{profth1}

\vspace{0.25cm}
G.N.Watson \cite[p.297 ]{W} obtained
\begin{equation}\label{2.19}
   \theta_{n} - \frac{1}{3} = \frac{1}{2} \int_{0}^{\infty} e^{- n x}
\left(- U^{\,\prime \prime } (x) - u^{\,\prime \prime } (x)\right) d x \ ,
 \ n \geq 0 \ ,
\end{equation}
\noindent
and \cite[p.300]{W}
\begin{equation}\label{2.20}
    U^{\,\prime \prime } (0) + u^{\,\prime \prime } (0) = - \frac{8}{135} \ .
\end{equation}

Integration by parts applied to \eqref{2.19} using \eqref{2.20} gives
the basic relations used for this proof:

\begin{align}\label{2.21}
 n  \left( \theta_{n} - \frac{1}{3}\right)   & =  \frac{4}{135} -
\frac{1}{2} \int_{0}^{\infty} e^{- n x}
\left(U^{\,\prime \prime \prime  } (x) + u^{\,\prime \prime \prime  } (x)\right) d x \ ,
 \ n \geq 0 \ , \\[0.3cm]  \label{2.22}
 \frac{\theta_{n} - \frac{1}{3}}{n} & =
\frac{1}{2} \int_{0}^{\infty} e^{- n x}
\left( \frac{4}{3} -  U^{\, \prime }\left(x\right) - u^{\, \prime} \left( x\right) \right) d x  \ ,
 \ n \geq 1 \ ,\\[3mm]
\label{2.24}
  \left(n + 1\right)  \left(\theta_{n} - \frac{1}{3} \right) -
    \frac{4}{135} &=
    \frac{1}{2} \int\limits_{0}^{\infty} e^{- n x} \Delta (x) d x \ , \
    n \geq 0 \ ,
\end{align}
where
\begin{equation}
  \label{eq:x3}
 \Delta(x) :=  -U^{\,\prime \prime } (x) -  u^{\,\prime \prime } (x) - U^{\,\prime \prime \prime  } (x) - u^{\,\prime \prime \prime  } (x) \ , \  x \geq 0 \ .
\end{equation}

We mention in passing that \eqref{2.22} leads immediately to the
representation \eqref{4} for the Ramanujan sequence  given by  J. Adell and P. Jodra \cite[(5), p.3]{Ad}.
It follows from \eqref{2.8}, \eqref{2.9} and Theorem \ref{thm4} that
\begin{gather}\label{2.17}
\lim\limits_{x\to +\infty}\left(u^{\, \prime }(x) +  U^{\, \prime }(x)\right) = 1 \  , \
 u^{\, \prime } \left( 0\right) +  U^{\, \prime } \left(0\right) = {4}/{3} \  , \ \ \ \
  u^{\, \prime } \left( x\right) +  U^{\, \prime } \left(x\right) > 0 \ , \  x > 0 \  .
\end{gather}
Theorem \ref{thm4} (see also Watson \cite[p.298]{W} and Alzer \cite[p.641]{Al}) implies $U^{\, \prime \prime  }(x)+u^{\, \prime \prime  }(x)<0,\ x>0,$ so that the function
\begin{gather}\label{2.23}
    G_{0} (x) := 3 \left[\frac{4}{3} -  U^{\, \prime }\left(x\right) - u^{\, \prime } \left( x\right)\right] \ , \ \ x \geq 0 \ ,
\end{gather}
increases from $0$ to $1$ on the positive half-line and in view of \eqref{2.19},
 \begin{gather*}
  \theta_{n} - \frac{1}{3} = \frac{1}{6} \int_{0}^{\infty} e^{- n x}
 d G_{0} (x) =  \frac{1}{6}\int_{0}^{1} t^{n}  d\left[1- G_{0} \left(\log \frac{1}{t}\right)\right]  \ ,
 \ n \geq 0 \ .
 \end{gather*}
 Since
 $ \theta_{n} >  0$, $n \geq 0$, and
 \begin{gather*}
 \sum_{m=0}^{n} C^{m}_{n} (-1)^{m} \theta_{k+m} = \sum_{m=0}^{n} C^{m}_{n} (-1)^{m}\left( \theta_{k+m} - \frac{1}{3} \right)=
\frac{1}{6}\int_{0}^{\infty}  e^{- k x }\left(1 - e^{ -  x}\right)^{n}  d G_{0} (x) > 0 \ ,
 \end{gather*}
for all $k \geq 0$ and $ n \geq 1$, we obtain the complete monotony of  $\{\theta_{n}\}_{n \geq 0}$ and the validity
of (\ref{4}) for $F (x):= 1- G_{0} \left(\log {1}/{x}\right)$, $0 < x \leq 1$, $F (0):=0$.

\vspace{0.15cm}
Furthermore, by writing \eqref{2.21} in the form
 \begin{align*}
1-\frac{135\,n}{4}\left(\theta_n-\frac{1}{3}\right)& =\frac{135}{8} \int_{0}^{\infty} e^{- n x}
\left(U^{\,\prime \prime \prime  } (x) + u^{\,\prime \prime \prime  } (x)\right) d x =  \int_{0}^{\infty} e^{- n x}  d G_{1}(x) \ ,
 \ n \geq 0 \ ,
\end{align*}
where
\begin{align*}
  G_{1}(x) & = 1 +  \frac{135}{8} \left( U^{\,\prime \prime } (x) + u^{\,\prime \prime } (x) \right)\ , \ x \geq 0 \ ,
 \end{align*}

\noindent
we obtain  the validity of \eqref{7} for $D (x):= 1- G_{1} (\log (1/x))$
  because $U^{\prime \prime \prime  }+u^{\prime \prime \prime  }$ is non-negative for all $x \geq 0$
by Theorem \ref{thm4} (see also Koumandos \cite[p.452]{K}).

\vspace{0.15cm}
Using \eqref{2.24} and again Theorem
 \ref{thm4} to see that $\Delta(x)$ is non-negative we conclude that
 \begin{gather*}
  \left(n + 1\right)  \left(\theta_{n} - \frac{1}{3} \right)  -     \frac{4}{135} =
   \frac{37}{270}\int\nolimits_{[0,1]} x^n
d G (x)   \ , \
    n \geq 0 \ ,
\end{gather*}
where
\begin{gather*}
 \frac{37}{45} G (e^{-x}) :=   3 \left[ U^{\, \prime }\left(x\right) + u^{\, \prime } \left( x\right) +
 U^{\,\prime \prime } (x) +  u^{\,\prime \prime } (x)-1\right] \ , \   0 \leq  x < +\infty\ ,
 \end{gather*}

\noindent
  $G (0) := G (0+0) =   0$, $G (1-0)  = G (1) = 1$ and $({37}/{45})e^{-x} G^{\, \prime } (e^{-x}) = 3 \Delta(x) > 0$ for all $x > 0$. Therefore $G$ is a probability distribution function  on $[0,1]$ and the sequence in question turns out to be completely monotone.
  The proof of Theorem \ref{th1} is now complete.

\section{\texorpdfstring{Proof of Theorem \ref{th3}}{Proof of Theorem 2.3}}

\vspace{0.25cm}
Note that \eqref{2.19} and \eqref{2.22} imply the following
integral representations for the functions dealt with in   \eqref{10}
and Theorem \ref{th3} (for $\alpha=1$),
 \begin{align}\label{2.26}
  \left(\theta_{1}- 1/3\right)  \sigma (z) & =
  \frac{1}{2} \int_{1}^{\infty}\frac{z}{t - z }
\frac{- U^{\,\prime \prime } \left( \log  t\right) - u^{\,\prime \prime } \left( \log t \right)}{t} d t \ ,  \ \ z \in \Lambda \ \ ,
  \\[0.4cm]
  \label{2.27}\left(\theta_{1}- 1/3\right)  \sigma_{1} (z) &  =
 \frac{1}{2} \int_{1}^{\infty}\frac{z}{t - z }
\frac{ \left(4/3\right) -  U^{\, \prime }\left( \log t \right) - u^{\, \prime } \left( \log t \right)}{t}  d t \ ,  \ \ z \in \Lambda \ \ ,
 \end{align}
and recall, using Theorem \ref{poly}, that we need to prove Theorem
\ref{th3} only for the case $\alpha=1$. The necessary and sufficient
condition for $\sigma_1$ to have that property is given in the
following result.

\begin{thmx}[(Corollary 1.1 {\cite{RSS}})]\label{xthbr}
  Let $f \in \hol_{1} (\Lambda)$. Then $f$ is universally starlike
if and only if there exists a probability measure
  $\mu$ on $[0, 1]$ such that
  \begin{gather}\label{8}
    \dfrac{f(z)}{z} =
       \exp{\ {
    {{\int\nolimits_{[ 0 , 1\, ]}}} \,  \log \frac{1}{1 - t z} \ d \mu (t)
    }} \ , \ \ z \in \Lambda \ .
  \end{gather}
\end{thmx}

\subsection{Auxiliary results}\label{s3}

\vspace{0.15cm}
The following lemma is from  \cite[Theorem 1.10, p.294]{RSS}.
\begin{lemx}
   Let $\varphi, \psi :  (0, 1) \to [0, +\infty )$ be two integrable
   functions on $[0,1]$  satisfying
\begin{gather}\label{lem01} \int_{0}^{1} \varphi (t) \ d t = \int_{0}^{1} \psi (t) \ d t > 0  \ , \ \
    \begin{vmatrix}
      \varphi (x_{2}) & \psi (x_{2}) \\
 \varphi (x_{1}) & \psi (x_{1}) \\
    \end{vmatrix} \geq 0 \ , \ 0 < x_{1} \leq  x_{2} < 1 \ .
\end{gather}
Then there exists a probability measure $\mu$ on $[0,1]$ such that
\begin{gather}\label{lem02}
 \int\nolimits_{[ 0 , 1\, ]} \frac{ d \mu (t)}{1 - t z} \ = \int\nolimits_{0}^{1} \dfrac{\varphi (t) }{1-t z} d t \Big/ \int\nolimits_{0}^{1} \dfrac{\psi (t)}{1-t z}  \ d t \ ,  \ \ z \in \Lambda \ .
\end{gather}
\end{lemx}

Lemma A allows us to prove the next statement.
\begin{lemma}\label{lemma1}
   Let $g:  (0, +\infty ) \to (0, +\infty )$ be twice continuously
   differentiable on $(0,\infty)$ and assume
   \begin{gather}\label{3.1}
         \begin{array}{ll}
 (a) \ \lim_{x \downarrow 0 } g (x) = 0  \ ,  &
(b) \ \int_{0}^{\infty} e^{-x} g (x) d x = 1 \ ,
 \\[0.2cm]
   (c)\ g^{\, \prime } (  x ) \geq 0\ , \ x > 0 \ ,        & (d) \   g^{\, \prime } (  x )^{2} - g^{\, \prime \prime} (  x )  g (  x ) \geq 0 \ , \ x > 0 \ .
      \end{array}
      \end{gather}
    Then there exists a probability measure $\mu$ on $[0,1]$ such that
 \begin{gather}\label{3.2}
   \int_{1}^{\infty}  \dfrac{ g ( \log t)/t    }{t-z}  \ d t   = \exp{\ {
    {{\int\nolimits_{[ 0 , 1\, ]}}} \, \log \frac{1}{1 - t z} \ d \mu (t)
    }} \ ,  \ \ z \in \Lambda \ .
 \end{gather}
\end{lemma}
\begin{proof}
   Denote $v (x) :=  g ( \log x)/x$, $x > 1$, and for $z \in \Lambda$ let
\begin{gather*}
    f (z) := z  \int_{1}^{\infty}  \dfrac{ g ( \log t)/t    }{t-z}  \ d t  =  z \int_{1}^{\infty} \frac{v (t)  \ d t }{t-z}  = \int_{1}^{\infty} \left[
-1 + \frac{t}{t-z}\right] v (t) \ d t  \ .
\end{gather*}
The properties \eqref{3.1}(a),(b) mean that $v \in L_1 ([1, +\infty))$
and $\lim_{x\to 1+0}v (x) = 0$, which implies that
for arbitrary $z \in \Lambda$ we have
\begin{gather*}
f^{\, \prime } (z) =  \int_{1}^{\infty} \frac{t v (t) }{(t-z)^{2}} d t = -  \int_{1}^{\infty}t v (t) d \frac{1}{t-z} =
\frac{v (1)}{1-z} +  \int_{1}^{\infty} \frac{\left(t v (t)\right)^{\, \prime } }{t-z} d t = \int_{1}^{\infty} \frac{\left(t v (t)\right)^{\, \prime } }{t-z} d t \ ,
\end{gather*}
and
\begin{gather*}
    \frac{z f^{\, \prime } (z)}{ f (z)} = \frac{\int\limits_{1}^{\infty} \dfrac{\left(t v (t)\right)^{\, \prime } }{t-z} d t}{\int\limits_{1}^{\infty} \dfrac{v (t)}{t-z}  \ d t} = \frac{\int\limits_{0}^{1} \dfrac{\left[(1/t) v^{\, \prime }(1/t) + v (1/t)\right]/t }{1-t z} d t}{\int\limits_{0}^{1}
\dfrac{v (1/t)/t}{1-t z}  \ d t} =
\frac{\int\limits_{0}^{1} \dfrac{\varphi (t) }{1-t z} d t}{\int\limits_{0}^{1} \dfrac{\psi (t)}{1-t z}  \ d t} \ \ ,
\end{gather*}
where
\begin{gather*}
 \varphi (x) := \left((1/x) v^{\, \prime }(1/x) + v (1/x)\right)/x \ , \ \
\psi (x) :=  v (1/x)/x \ , \  \ 0 < x < 1 \ ,
\end{gather*}
and
\begin{gather*}
   \int_{0}^{1} \varphi (t) \ d t = \int_{0}^{1} \psi (t) \ d t = \int_{1}^{\infty} \frac{v (x)}{x} d x =
\int_{1}^{\infty} \frac{g (\log x)}{x^{2}} d x =  \int_{0}^{\infty} \frac{g ( x)}{e^{x}} d x = 1 \ .
\end{gather*}
Moreover, it follows from \eqref{3.1}(d) that the function $g^{\, \prime }(\log x) /  g (\log x)$ is non-increasing on $(1, +\infty)$ and since
\begin{gather*}
   1 + \frac{x v^{\, \prime }(x)}{ v (x)} = \frac{\dfrac{d}{d x} \left[x v(x)\right] }{ v (x)} =
  \frac{ \dfrac{d}{d x } \left[x \cdot \dfrac{g ( \log x )}{x}\right]}{ \dfrac{g ( \log x )}{x} } \ \ = \frac{\dfrac{g^{\, \prime }( \log x )}{x} }{ \dfrac{g ( \log x )}{x}} = \frac{g^{\, \prime }( \log x )}{g ( \log x )}
\end{gather*}
the function $x v^{\, \prime }(x) /  v (x)$ also does not increase on $(1, +\infty)$. This means that
for arbitrary $ 0 < x_{1} \leq  x_{2} < 1 $ we have
\begin{gather*}
    0 \leq \frac{1}{x_1 x_2} \begin{vmatrix}
 v (1/x_1)   & (1/x_1) v^{\, \prime }(1/x_1) \\
 v (1/x_2)  & (1/x_2) v^{\, \prime }(1/x_2)   \\
\end{vmatrix} = \begin{vmatrix}
 v (1/x_1)/x_1   & (1/x_1^{2}) v^{\, \prime }(1/x_1) \\
 v (1/x_2)/x_2  & (1/x_2^{2}) v^{\, \prime }(1/x_2)   \\
\end{vmatrix} = \begin{vmatrix}
 \psi (x_1 ) &\varphi (x_1)\\
  \psi (x_2 )  & \varphi(x_2)  \\
\end{vmatrix} \ .
\end{gather*}
Lemma A guarantees the  existence of a probability measure $\mu$ on $[0,1]$ such that
\begin{gather*}
 \frac{z f^{\, \prime } (z)}{ f (z)} = \int\nolimits_{[ 0 , 1\, ]} \frac{ d \mu (t)}{1 - t z}\ ,  \ \ z \in \Lambda \ .
\end{gather*}
Since
\begin{gather*}
    \frac{\dfrac{d}{d z} \dfrac{f (z)}{z}}{ \dfrac{f (z)}{z}} = \frac{ \dfrac{zf^{\, \prime } (z) - f (z)}{z^{2}}}{\dfrac{f (z)}{z}} = \frac{ f^{\, \prime } (z)}{ f (z)} - \frac{1}{z} = \int\nolimits_{[ 0 , 1\, ]} \frac{d}{d z} \log \frac{1}{1- t z} \ d \mu (t)
\end{gather*}
we can integrate this equality from $0$ to $z \in \Lambda$ and  obtain
\begin{gather*}
    \log  \frac{f (z)}{z} -  \log  f^{\, \prime } (0) = \int\nolimits_{[ 0 , 1\, ]} \log \frac{1}{1 - t z} \ d \mu (t) \ ,
\end{gather*}
where $f^{\, \prime } (0) = 1$ by virtue of \eqref{3.1}(b). Lemma~\ref{lemma1} is proved.
\end{proof}

\subsection{The proof}\label{s4}

\vspace{0.15cm}
By Theorem \ref{xthbr} the statement of  Theorem~\ref{th3} for $\alpha =1$ means that for $\sigma_1$
there exists a probability measure   $\mu$ on $[0, 1]$ such that
  \begin{gather}\label{fth31}
    \dfrac{\sigma_{1}(z)}{z} =
        \exp{\ {
    {{\int\nolimits_{[ 0 , 1\, ]}}} \,  \log \frac{1}{1 - t z} \ d \mu (t)
    }} \ , \ \ z \in \Lambda \ ,
  \end{gather}
where in accordance with  \eqref{2.27},
\begin{gather*}
 2 \left(\theta_{1}- 1/3\right) \frac{\sigma_{1} (z)}{z}   =
  \int_{1}^{\infty}\frac{1}{t - z }
\frac{ g \left( \log t \right)}{t}  d t \ ,  \ \ z \in \Lambda \ \ , \\
 g (x) := \frac{4}{3}-  U^{\, \prime } \left( x\right) -  u^{\, \prime } \left( x\right)  \ , \ \ x > 0 \ ,
\end{gather*}
and in view of \eqref{2.22}
\begin{gather*}
 \int_{0}^{\infty} e^{-x} g (x) d x  =
 \int_{0}^{\infty} e^{-x} \left[4/3 -U^{\, \prime } \left( x\right) -  u^{\, \prime } \left( x\right)\right] \ d x =2 \left(\theta_{1}- 1/3\right) \ .
\end{gather*}
Moreover, Theorem~\ref{thm4} and \eqref{2.17} imply that for arbitrary $x > 0$
\begin{align*}
 & g (x)  =  \frac{4}{3} -  U^{\, \prime } \left( x\right) -  u^{\, \prime } \left( x\right) \in \left[ 0 \ , \  \frac{1}{3}  \right] \ ,  \ \
     & g^{\, \prime} (x)   = -  U^{\, \prime \prime  } \left( x\right) -  u^{\, \prime \prime  } \left( x\right) > 0 \ , \\ &  g^{\, \prime \prime} (x) = -  U^{\, \prime \prime \prime   } \left( x\right) -  u^{\, \prime \prime \prime   } \left( x\right) < 0  \ , \ \  &  g (0)   =  \frac{4}{3} -  U^{\, \prime } \left(0\right) -  u^{\, \prime } \left(0\right) = 0 \   .
\end{align*}
Thus,
\begin{gather*}
 g^{\, \prime } (  x )^{2} - g^{\, \prime \prime} (  x )  g (  x ) >  0 \ , \ x > 0 \ ,
 \end{gather*}
and Lemma~\ref{lemma1} yields the validity of \eqref{fth31}.


\newpage

\section{\texorpdfstring{Proof of Theorem \ref{thm4}}{Proof of Theorem 3.1}}\label{sect5}

\subsection{An algorithm}\label{sect51}

In this section we present a general algorithm which gives sufficient
conditions for inequalities of the type described in Theorem
\ref{thm4}. It is dealing with exponential polynomials on $\Bb{R}^+$.

\begin{definition}\label{def2}
 Let
$$ f(x):=\sum_{k=0}^{m} P_k(x) e^{k x},$$
where the $P_k$ are real polynomials of exact degree $n_k$. Then we call $f$ an
{\em exponential polynomial } of order $m$ and (multi-)degree
  $\{n_0,\dots,n_m\}$.
\end{definition}

\begin{remark}
  A polynomial $P(x)=\sum_{j=0}^{n}a_j x^j$ is said to be of exact
  degree $n$ if $a_n\neq0$. If $P\equiv0$ then we say it is of
  (exact) degree $-1$.
\end{remark}

\begin{thm}
 \label{thm5}
 Let
$$ f(x):=f_0(x)=\sum_{k=0}^{m} P_k(x) e^{k x}=\sum_{k=0}^{\infty}a_k x^k
$$
be an exponential polynomial of order $m$ and degree
$\{n_0,\dots,n_m\}$. Let
\begin{equation}
  \label{eq:6}
f_{k+1}(x):=f_k^{(n_k+1)}(x)e^{-x},\quad k=0,\dots,m-1,
\end{equation}
and assume that
\begin{equation}
  \label{eq:7}
f_k^{(s)}(0)\geq0,\quad s=0,\dots,n_k,\ k=0,\dots,m.
\end{equation}
Then all Taylor coefficients $a_k,\ k\geq 0,$ of $f(x)$ are non-negative. In
particular,
$f(x)\geq0,\ x\geq 0$.
\end{thm}

\begin{remark}
 Note that there are only {\em finitely} many, namely
$$\mu(f):=\sum_{k=0}^{m}(n_k+1),$$ conditions
to be tested (which involve only the first $\mu(f)$
coefficients $a_k$ of $f$) to draw the conclusion for {\em all }
coefficients of $f$.
\end{remark}

\begin{proof}
The proof runs by mathematical induction. First note that for any
exponential polynomial
$$
h(x)=\sum_{k=0}^{m}P_k(x)e^{kx}, \quad \mbox{degree}(h)=
\{n_0,\dots,n_m\},
$$
we have
$$
h'(x)= P_0'(x)+\sum_{k=1}^{m}(P_k'(x)+k P_k(x))e^{kx},
$$
which is an exponential polynomial of order $m$ and degree
$\{n_0-1,n_1,\dots,n_m\}$. Further, if  $n_0=0$,
the function $h'(x) e^{-x}$ is an exponential polynomial of order
$m-1$ and degree $\{n_1,\dots,n_m\}$.

We begin with the case $m=0$. Then we have $f=P_0$ with degree
$\{n_0\}$. In this case the conditions \eqref{eq:7} just say
$$P_0^{(s)}(0)\geq0,\quad s=0,\dots,n_0,
$$
which means that all coefficients of $P_0$ (and therefore of $f$) are
non-negative. This settles the case $m=0$.

Now assume that the theorem is valid for some $m-1\geq0$, and let $f=f_0$
be as in the statement of the theorem. The way the function $f_1$ is
defined
it is clear that it is an exponential polynomial of order $m-1$, and
degree $\{n_1,\dots,n_m\}$ and the conditions \eqref{eq:7}, applied to
$f_1$ instead of $f_0$, show, by our assumption that the theorem is
correct for functions of degree $m-1$, that $f_1$ has all of it's
Taylor coefficients non-negative, which implies that the Taylor
coefficients of
\begin{equation}
  \label{eq:8}
f_0^{(n_0+1)}(x)=e^x f_1(x)
\end{equation}
are also all non-negative. The conditions \eqref{eq:7} concerning $f_0$
now say  that the remaining first coefficients of $f$, namely
$a_s =f_0^{(s)}(0)/s!$, \ $s=0,\dots,n_0$, are non-negative as well. This
completes the proof.\end{proof}

When it comes to the application of this theorem we have to keep the
following facts in mind:
\begin{enumerate}
\item This algorithm is particularly suited for cases when the
  polynomials $P_k$ have rational coefficients only since then all the
  numbers $f^{(s)}_k(0)$ are rational and therefore their calculation
  via a computer algebra program is exact and no numerical problem, f.i. with
  cancelation, occurs.
\item Given an exponential polynomial it is not really necessary to
  know its order or multi-degree to begin with: the algorithm can
  decide, when properly implemented, by itself what to do next
  (differentiate once more or go to the next $f_k$. In particular it
  can stop as soon as one of the numbers  $f^{(s)}_k(0)$ turns out to
  be negative, which can save machine time.
\end{enumerate}

\subsection{The proof}\label{thep}

The proof of Theorem \ref{thm4} will be completely computer based,
using the algorithm just described. All coefficients in these cases are
rational, actually integers, so there is no numerical problem. The algorithm has been
programmed using Mathematica version 9.0 and run on a laptop
computer. Computation time was a few seconds for each of the three
cases to be verified for Theorem \ref{thm4}.

The resulting numbers  $f^{(s)}_k(0)$ are collected in one single vector
$\lambda(f)$ with
$\mu(f)$ entries, listed in their natural order as they are being
calculated by the algorithm. If $\lambda(f)$ turns out to be
non-negative then the case under
consideration has been settled.

\vspace{6mm}
{\bf 6.2.1. \  Case 1:  ${\mathbf{ (U'+u')+(U''+u'')>0}}$.}\

\vspace{0.25cm}
Using \eqref{2.12}, \eqref{2.13} we find for
\begin{align*}
R_1(x)&:=\frac{H(x)}{H(x)-1}+\frac{h(x)}{h(x)-1}-\frac{H(x)}{(H(x)-1)^3}-
\frac{h(x)}{(h(x)-1)^3}\\
&=\frac{x^2}{((e^x-1-x)^3 (1-e^x(1-x))^3}S_1(x),
\end{align*}
where
\begin{align*}
S_1(x)= & \ (-2-x)
+(8-3x-3x^2)e^x+(-14+9x-6x^2-5x^3)e^{2x}\\
&+(16+18x^2-2x^4)e^{3x}
+(-14-9x-6x^2+5x^3)e^{4x}\\
&+(8+3x-3x^2)e^{5x}
+(-2+x)e^{6x}.
\end{align*}

\noindent
So $S_1$ is an exponential polynomial of order 6 and degree
$\{1,2,3,4,3,2,1\}$.
Application of the algorithm produces the vector
\begin{eqnarray*}
\lambda(S_1)&=&
(0,0,0,0,0,0,0,0,0,0,72240,1155840,9557760,56267040,
\\ & &271084224,880843680, 2475629568,6343909632,
\\  & & 1533939393792,20392197120,25057382400,29561241600,\\ & & 4478976000) \ ,
\end{eqnarray*}
which proves that the desired inequality is valid.

\vspace{6mm}

{\bf 6.2.2. \  Case 2: ${\mathbf{ -(U''+u'')-(U'''+u''')>0}}$.} \

\vspace{0.25cm}
Here we have to show that

\begin{align*}
R_2(x)&:=\frac{H(x)}{(H(x)-1)^3}+\frac{h(x)}{(h(x)-1)^3}- H(x)\frac{1+2H(x)}{(H(x)-1)^5}-h(x)\frac{1+2h(x)}{(h(x)-1)^5}\\
&=\frac{x^2(e^x-1)^2}{((e^x-1-x)^5 (1-e^x(1-x))^5}S_2(x),
\end{align*}
where
\begin{align*}
S_2(x)= & \  (-4-x)
+(24-15x-5x^2)e^x
+(-64+70x-60x^2-50x^3-20x^4-4x^5)e^{2x}\\
&+(104-91x+285x^2+100x^3+20x^4-x^5-x^6)e^{3x}
+(-120-440x^2)e^{4x}\\
&+(104+91x+285x^2-100x^3+20x^4+x^5-x^6)e^{5x}
\\&+(-64-70x-60x^2+50x^3-20x^4+4x^5)e^{6x}
+(24+15x-5x^2)e^{7x}
+(-4+x)e^{8x}.
\end{align*}

So $S_2$ is an exponential polynomial of order 8 and degree
$\{1,2,5,6,2,6,5,2,1\}$.
Application of the algorithm produces the vector
\begin{eqnarray*}
\lambda(S_2)&=&
(0,0,0,0,0,0,0,0,0,0,0,0,0,0,0,0,1095494400,38342304000,\\ & & 718413696000, 8922167654400,
85789518796800,686634000998400,\\ & & 4108040955648000, 21277519458048000,98491821821245440,
\\ & & 417993857883463680,1659729058910208000, 6264125727645450240,\\ & & 22744955668622376960,
57435249160046592000, 138673044884876820480,
\\ & & 324272107555238707200, 741041088684097536000, 1665009811944898560000,
 \\ & & 3693054970331136000000,4415481367363584000000, 5133351192625152000000,
\\ & & 5850215720681472000000, 716770887598080000000) 
\end{eqnarray*}
which proves the claim as all entries are non-negative.

\vspace{6mm}
{\bf 6.2.3. \ Case 3: ${\mathbf{ (U'''+u''')>}}0$.}\

\vspace{0.25cm}
Here we have to show that

\begin{align*}
R_3(x):=
H(x)\frac{1+2H(x)}{(H(x)-1)^5}+h(x)\frac{1+2h(x)}{(h(x)-1)^5}
=\frac{x(e^x-1)^3}{((e^x-1-x)^5 (1-e^x(1-x))^5}S_3(x)>0,
\end{align*}
where
\begin{align*}
  S_3(x):=& \ 1-2x+(-5+20x+10x^3+5x^4+x^5)e^x
\\&+
(9-72x-70x^3-30x^4-11x^5-2x^6)e^{2x}+(-5+130x+160x^3+25x^4-10x^5)e^{3x}
\\ &+(-5-130x-160x^3+25x^4-10x^5)e^{4x}
+(9+72x+70x^3-30x^4+11x^5-2x^6)e^{5x}\\&+
(-5-20x-10x^3+5x^4-x^5)e^{6x}
+(1+2x)e^{7x}.
\end{align*}

So $S_3$ is an exponential polynomial of order 7
 and degree
$\{1,5,6,5,5,6,5,1\}$.
Application of the algorithm produces the vector
\begin{eqnarray*}
\lambda(S_3) &= &
(0,0,0,0,0,0,0,0,0,0,0,0,0,0,0,0,115315200,3863059200,70457587200,\\
& & 927826099200, 9830767564800,8631514316800,615374090956800,\\
& & 83729093713049600,20168695176376320,99183876729477120,450524284521338880,\\
& & 1915432618475059200,5792081300977213440,16127157987099279360, \\
& & 41953781738132766720, 103330763975294484480, 243753521061983846400,   \\
& &556095351762151833600,1236678576792676761600,  1461058224846520320000,\\
& &  1642128742165708800000,1795208480980992000000, 1936007205617664000000,\\
& & 2073220384555008000000, 2209799770472448000000, 136527788113920000000)
\end{eqnarray*}
which is also non-negative. The proof is complete.



\affiliationone{Andrew  Bakan\\
  Institute of Mathematics\\
   National Academy of
Sciences of Ukraine\\
01601 Kyiv \\
   Ukraine
   \email{andrew@bakan.kiev.ua} }
\affiliationtwo{Stephan Ruscheweyh\\
   Institut f\"ur Mathematik \\
    Universit\"at W\"urzburg \\
97074 W\"urzburg, Germany
\email{ruscheweyh@mathematik.uni-wuerzburg.de}}
\affiliationthree{Luis Salinas \\
Departamento de Inform\'atica, UTFSM \\
 Valpara\'\i{}so, Chile
 \email{luis.salinas@usm.cl}}
\end{document}